\documentclass[10pt,bezier]{article}
\usepackage{amsmath, amsthm, amscd, amsfonts, amssymb, graphicx, color}
\usepackage[bookmarksnumbered, plainpages]{hyperref}
\usepackage{tikz}
\usetikzlibrary{arrows}
\date{ }

\normalsize

\newtheorem{lemma}{Lemma}
\newtheorem{theorem}[lemma]{Theorem}

\numberwithin{equation}{section}

\newtheorem{remark}[lemma]{Remark}
\newtheorem*{question}{Question}

\title{\bf The influence of cut vertices and eigenvalues on character graphs of solvable groups}

\author{\bf Roghayeh Hafezieh$^a$, Mohammad Ali Hosseinzadeh$^b$,\\
{\bf Samaneh Hossein-Zadeh$^b$, Ali Iranmanesh$^b$}\footnote{Corresponding Author: Ali Iranmanesh\ \hspace{15 cm}
E
E-mail addresses:
roghayeh@gtu.edu.tr (R. Hafezieh),
ma.hosseinzadeh@modares.ac.ir (M. A. Hosseinzadeh), s.hosseinzadeh@modares.ac.ir (S. Hossein-Zadeh),
iranmanesh@modares.ac.ir (A. Iranmanesh)}
\\[2mm]
$^a$\small Department of Mathematics, Gebze Technical University, P.O. Box 41400, Gebze, Turkey.\\
$^b$\small Department of Mathematics, Faculty of Mathematical
Sciences,\\ \small Tarbiat Modares University, P.O. Box 14115-137, Tehran, Iran.}

\begin{document}
\maketitle
\begin{abstract}
Given a finite group $G$, the \textit{character graph}, denoted by $\Delta(G)$, for its irreducible character degrees is a graph with vertex set $\rho(G)$ which is the set of prime numbers that divide the irreducible character degrees of $G$, and with $\{p,q\}$  being an edge if there exist a non-linear $\chi\in {\rm Irr}(G)$ whose degree is divisible by $pq$. In this paper, we discuss the influences of cut vertices and eigenvalues of $\Delta(G)$ on the group structure of $G$. Recently, Lewis and Meng proved the character graph of each solvable group has at most one cut vertex. Now, we determine the structure of character graphs of solvable groups with a cut vertex and diameter $3$. Furthermore, we study solvable groups whose character graphs have at most two distinct eigenvalues. Moreover, we investigate the solvable groups whose character graphs are regular with three distinct eigenvalues. In addition, we give some lower bounds for the number of edges of $\Delta(G)$.
\vskip 3mm


\vskip 3mm


\end{abstract}

\section{Introduction}
Let $G$ be a finite group. We consider the set of irreducible complex characters of $G$, namely ${\rm Irr}(G)$, and the related degree set ${\rm cd}(G)=\{\chi(1) : \chi\in {\rm Irr}(G)\}$. Let $\rho(G)$ be the set of all primes which divide some character degrees of $G$. There is a large literature which  is  devoted  to  study  the ways  in  which  one  can  associate  a  graph  with  a  group,
for  the  purpose  of  investigating  the  algebraic  structure using  properties  of  the  associated  graph. One of these graphs is the character graph of $G$ which is denoted by $\Delta(G)$. It is an undirected graph with vertex set $\rho(G)$, where $p,q\in\rho(G)$ are joined by an edge if there exists an irreducible character degree
$\chi(1)\in {\rm cd}(G)$ which is divisible by $pq$.

It is an interesting subject to determine which finite simple graphs can occur as the
character graphs of finite groups. This question has attracted many researchers over the
years. For more information we refer to see the survey paper ~\cite{L} in which the author discussed many remarkable connections between this graph and some other graphs associated to irreducible character degrees, by analysing properties of these graphs for arbitrary positive integer subsets.

Since cut vertices, perfect matchings, and the number of edges affect the structure of the character graph, and in particular affect the structure of the group itself, in this paper, we will discuss the influences of cut vertices and perfect matchings on the character graphs of finite solvable groups.

Given a finite group $G$, in the second section, considering two cases with respect to solvability of $G$, we give a lower bound for the number of edges and an upper bound for the domination number of $\Delta(G)$. In addition, we give an example which implies that these bounds are best possible. Theorem ~\ref{NE} is the main theorem of this section.

In the third section, we prove that, for a solvable group $G$ whose $\rho(G)$ has even number of elements, $\Delta(G)$ has a perfect matching. Also, if $\Delta(G)$ has odd number of vertices, then each of its vertex-deleted subgraphs has a perfect matching. Furthermore, we discuss the influence of number of eigenvalues of $\Delta(G)$ on the group itself. Theorem ~\ref{reg} is the main theorem of this section.

In ~\cite{EI}, the authors discussed the necessary and sufficient conditions for $\Delta(G)$ to be Hamiltonian. As a graph with a cut vertex does not have a hamiltonian cycle, it is interesting to discuss the influence of such vertices on the structure of the graph $\Delta(G)$ and the structure of the group itself. Hence in section four, considering $G$ as a solvable group, we examine such vertices in $\Delta(G)$. We show that the character graph of any solvable group does not have more than one cut vertex and then, we determine the structure of such character graphs with one cut vertex and diameter $3$. Theorems ~\ref{thm: cut3} and ~\ref{cut} are the main theorems of this section.

In the rest of this section, we mention some notations which are used in this paper. Let $\mathcal{G}$ be a graph. The vertex set and the edge set of $\mathcal{G}$ are denoted by $V(\mathcal{G})$ and $E(\mathcal{G})$, respectively. The order of $\mathcal{G}$ is the number of its vertices. For a vertex $v\in V(\mathcal{G})$, by $N(v)$ we mean the neighbors of $v$ and so $deg(v) = |N(v)|$. The graph $\mathcal{G}-v$ is obtained by removing $v$ from $\mathcal{G}$. The graph $\mathcal{G}$ is called \textit{$k$-regular}, if the degree of each vertex is $k$. If $\mathcal{G}$ is neither empty nor complete, then it is said to be \textit{strongly regular} with parameters $(v,k,\lambda,\mu)$ if $|V(\mathcal{G})|=v$, $\mathcal{G}$ is $k$-regular, any two adjacent vertices of $\mathcal{G}$ have $\lambda$ common neighbors and any two non-adjacent vertices of $\mathcal{G}$ have $\mu$ common neighbors. We denote the complete graph and the cycle of order $n$, by $K_n$ and $C_n$, respectively. The \textit{distance} between two vertices $u$ and $v$ is the minimum length of paths with endpoints $u$ and $v$ which is denoted by $d(u,v)$. A path (cycle) containing all vertices of the graph $\mathcal{G}$ is called a \textit{Hamiltonian path (cycle)}. A graph $\mathcal{G}$ is called \textit{Hamiltonian}, if it has a Hamiltonian cycle. Let $n(\mathcal{G})$ be the number of connected components of $\mathcal{G}$. The {\it connectivity} of $\mathcal{G}$, $\kappa(\mathcal{G})$, is equal to the minimum number
of vertices in a set $S$ whose deletion results in $n({\mathcal{G}} - S) > n(\mathcal{G})$. Hence, if $\mathcal{G}$ is a connected graph, then $\kappa({\mathcal{G}}) \geq 1$. If the set $S$ contains only one vertex such as $v$, then we say $v$ is a \textit{cut vertex}. A \textit{cut edge} is an edge of a graph whose deletion increases its number of connected components. An {\it independent set} in $\mathcal{G}$ is a set of vertices of $\mathcal{G}$, no two of which are adjacent. An independent set in $\mathcal{G}$ is called {\it
maximum} if $\mathcal{G}$ contains no larger independent set. The cardinality of a maximum independent set in
$\mathcal{G}$ is called the {\it independent number}, denoted by $\alpha(\mathcal{G})$.
A \textit{clique} in $\mathcal{G}$ is a complete induced subgraph of $\mathcal{G}$. The \textit{adjacency matrix} of $\mathcal{G}$ is the $0-1$ matrix $A$ indexed by $V(\mathcal{G})$, where $A_{xy} = 1$ whenever there is an edge between $x$ and $y$ in $\mathcal{G}$ and $A_{xy} = 0$ otherwise. The eigenvalues of the adjacency matrix of the graph $\mathcal{G}$ are called {\it eigenvalues} of $\mathcal{G}$. Notice that a graph of order $n$ has $n$ eigenvalues.
The \textit{complement} $\bar{\mathcal{G}}$ of $\mathcal{G}$ is the graph whose vertex
set is $V(\mathcal{G})$ and whose edges are the pairs of non-adjacent vertices of $\mathcal{G}$.
A set $D\subseteq V(\mathcal{G})$ is called {\it dominating} if every vertex $w\in V(\mathcal{G})\setminus D$ is adjacent to a vertex of $D$. The {\it domination number} $\gamma (\mathcal{G})$ of a graph $\mathcal{G}$ is the size of a smallest dominating set of $\mathcal{G}$. Other notations are standard and taken mainly from ~\cite{BM,BH,IS}.

\section{Some Structural Properties of Character Graphs}

The character graph which has been studied extensively
over the last twenty years, has an important role to study the
character degrees of a group, see ~\cite{L}. In this section, we discuss some
combinatorial properties of the character graph, in
particular, we obtain some bounds for the number of edges
and the domination number of this graph, considering two
different cases, whether the group is solvable or not.

\begin{lemma}\label{P}
Let $G$ be a finite group and let $\pi$ be a subset of $\rho(G)$.
\begin{itemize}
\item[(i)] (Moret\'{o} and Tiep's Condition, ~\cite{MT}) If
$|\pi|\geq 4$, then there exists an irreducible character whose
degree is divisible by at least two primes from $\pi$.
\item[(ii)](P\'{a}lfy's Condition ~\cite{PP})  In particular, if $G$ is solvable and $|\pi|\geq 3$, then there exists an irreducible character whose
degree is divisible by at least two primes from $\pi$.
\end{itemize}
\end{lemma}

\begin{theorem}{\label{NE}}
Let $G$ be a finite group whose character graph has $n$ vertices and $m$ edges.
\begin{itemize}
    \item [(a)] If $G$ is a solvable group, then $m \geq \frac{n}{2}\left( \frac{n}{2} - 1\right)$ and $\gamma(\Delta(G))\leq 2$. In particular, if $\Delta(G)$ is connected, then it has a Hamiltonian path.
    \item [(b)] If $G$ is a non-solvable group, then $m\geq\frac{n}{2}\left( \frac{n}{3} - 1 \right)$ and $\gamma(\Delta(G))\leq 3$.
\end{itemize}
\end{theorem}

\begin{proof}

(a) Since $G$ is a solvable group, Lemma ~\ref{P} implies that the
complement of the graph $\Delta(G)$ is triangle free. On the
other hand, Tur\'{a}n's Theorem ~\cite[Theorem 12.3]{BM} verifies that
$|E(\overline{\Delta(G)})|\leq \frac{n^2}{4}$, so
${{n}\choose{2}}- m \leq \frac{n^2}{4}$. Hence $m \geq
\frac{n}{2}\left( \frac{n}{2} - 1\right)$, as we claimed. As
$\gamma(\mathcal{G})=1$ for a complete graph $\mathcal{G}$, we
may consider the case where $\Delta(G)$ is non-complete. Let $p$
and $q$ be two non-adjacent vertices of $\Delta(G)$. Lemma
~\ref{P} implies that every other vertices of $\Delta(G)$ is
adjacent to $p$ or $q$, therefore $\{p, q\}$ is a domination set
for $\Delta(G)$, which verifies that $\gamma(\Delta(G))\leq 2$.
Consider the case where $\Delta(G)$ is a connected graph. By Lemma ~\ref{P}, it is easy to see that $\alpha(\Delta(G))
\leq \kappa(\Delta(G)) + 1$. With respect to the results proved
by Chv\'{a}tal and Erd\"{o}s in ~\cite{CE}, we deduce that
$\Delta(G)$ has a Hamiltonian path.

(b) By Lemma ~\ref{P}, for a nonsolvable group $G$, the complement of the graph
$\Delta(G)$ contains no $K_4$. So, using Tur\'{a}n's Theorem ~\cite[Theorem 12.3]{BM} we have
$|E(\overline{\Delta(G)})|\leq \frac{2}{3}\left( \frac{n^2}{2}
\right)$. Hence ${{n}\choose{2}}- m \leq \frac{n^2}{3}$, which gives the result. Similar to
previous part and with respect to Lemma ~\ref{P}, one can see that
in this case $\gamma(\Delta(G))\leq 3$.
\end{proof}

\begin{remark}

1. It should be mentioned that solvability plays an important role in part $(a)$ of Theorem ~\ref{NE}.
It is famous that ${\rm cd}(PSL(2,2^{n}))=\{1,2^{n},2^{n}-1,2^{n}+1\}$. Also $|\rho(PSL(2,2^{n}))|=3$ if and only if
$n\in\{2,3\}$ (see ~\cite{Hu}). Now, let $H=PSL(2,2^{n})$, for $n\in\{2,3\}$ and $K$ be a non-abelian $p$-group, where $p > 7$ is a prime number.
Let $G=H\times K$, which is a non-solvable group. It is easy to see that $\Delta(G)$ is isomorphic to the star graph mentioned in Figure ~\ref{fig:11},
which is connected but it has no Hamiltonian path.

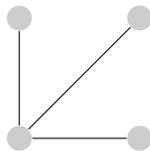
\begin{figure}[h]
\begin{center}
\begin{tikzpicture}
  [scale=.8,auto=left,every node/.style={circle,fill=black!20}]
  \node (n1) at (1,1)  {};
  \node (n2) at (1,3)  {};
  \node (n3) at (3,1)  {};
  \node (n4) at (3,3)  {};
  \foreach \from/\to in {n1/n2,n1/n3,n1/n4}
    \draw (\from) -- (\to);
\end{tikzpicture}
\end{center}
\caption{Star graph.}
~\label{fig:11}
\end{figure}

2. The upper bounds for $\gamma(\Delta(G))$, mentioned in Theorem
~\ref{NE}, are best possible. As an example of part $(a)$, one may
consider a solvable group $G$ of type five in ~\cite{ML2} whose
character graph is the graph with two connected components
in Figure \ref{fig:12}, while for part $(b)$, $PSL(2,16)$, whose
character graph is the graph with three connected components in
Figure \ref{fig:12} satisfies the equalities.
\end{remark}

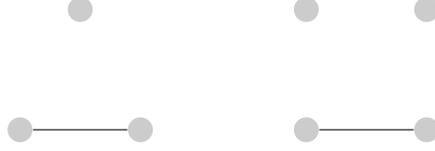
\begin{figure}[h]
\begin{center}
\begin{tikzpicture}
  [scale=.8,auto=left,every node/.style={circle,fill=black!20}]
  \node (n1) at (1,1)  {};
  \node (n2) at (3,1)  {};
  \node (n3) at (2,3)  {};

  \foreach \from/\to in {n1/n2}
    \draw (\from) -- (\to);
\end{tikzpicture}
\quad\quad\quad\quad\quad
\begin{tikzpicture}
  [scale=.8,auto=left,every node/.style={circle,fill=black!20}]
  \node (n1) at (1,1)  {};
  \node (n2) at (3,1)  {};
  \node (n3) at (1,3)  {};
  \node (n4) at (3,3)  {};

  \foreach \from/\to in {n1/n2}
    \draw (\from) -- (\to);
\end{tikzpicture}
\end{center}
\caption{Character graphs of a solvable and a nonsolvable group.}
~\label{fig:12}
\end{figure}

\section{Perfect Matchings and Eigenvalues of Character Graphs}

A \textit{matching} in a graph $\mathcal{G}$ is any set of edges $M\subseteq E(\mathcal{G})$ such that every vertex in
$\mathcal{G}$ is incident with at most one edge from $M$. The matching $M$ is called a \textit{perfect matching}, if every vertex of
$\mathcal{G}$ is incident with exactly one edge from $M$. A graph is called \textit{hypomatchable} if each of its vertex-deleted subgraphs has a perfect matching. If no vertex of a graph has three pairwise non-adjacent neighbors, it is called \textit{claw-free}. In the next theorem, we prove the existence of a perfect matching in character graphs of solvable groups under special conditions.

\begin{theorem}\label{pmatching}
Let $G$ be a solvable group whose character graph is connected. If the order of $\Delta(G)$ is even, then $\Delta(G)$ has a perfect matching. If $\Delta(G)$ has no cut vertex and its order is odd, then $\Delta(G)$ is hypomatchable. This property is not necessarily true for a nonsolvable group.
\end{theorem}

\begin{proof}
Let $G$ be a solvable group whose character graph is connected and it has even number of vertices. Lemma ~\ref{P}, implies that $\Delta(G)$ is claw-free. A claw-free connected graph with an even number of vertices has a perfect matching (see ~\cite{Las,S}), so does $\Delta(G)$. If $\Delta(G)$ has no cut vertex and its order is odd, then for each vertex of $\Delta(G)$ such as $v$, the graph $\Delta(G)-v$ is a connected claw-free graph whose order is even. Thus, $\Delta(G)-v$ has a perfect matching and so, $\Delta(G)$ is hypomatchable. On the other hand, consider the nonsolvable group $G=A_{5}\times E$, where $E$ is an extra-special group of order $p^3$ for some prime $p\not\in \{2,3,5\}$. One can see that $\Delta(G)$ has no perfect matching.
\end{proof}

In the sequel, we study solvable groups whose character graphs have at most three distinct eigenvalues.

\begin{theorem}\label{dist}
Let $G$ be a solvable group. The number of distinct eigenvalues of $\Delta(G)$ is equal to $1$ if and only if $G$ has one of the following properties:
\begin{itemize}
\item[(1)] $G\simeq A\times P$, where $A$ is an abelian group and $P$ is a non-abelian $p$-group, for a prime $p$.
\item[(2)] $G$ has a normal abelian subgroup $N$ whose index is a power of a prime number.
\item[(3)] $G$ is either a group of type two or three mentioned in ~\cite{ML2}.
\item[(4)] $G$ is a group of type one mentioned in ~\cite{ML2} where $|G/F(G)|$ is a prime power.
\item[(5)] $G$ is a group of type four mentioned in ~\cite{ML2} where in the sense of ~\cite{ML2}, $|E(G)/F(G)|$ and $m$ are distinct prime powers.
\item[(6)] $G$ is a group of type five mentioned in ~\cite{ML2} where in the sense of ~\cite{ML2}, $2^{a}+1$ is a prime power.
\end{itemize}
 Furthermore, this number is equal to $2$ if and only if $\Delta(G)= K_n$, for some $n\geq 2$.
\end{theorem}

\begin{proof}
A graph with only one eigenvalue is edgeless, so its eigenvalues are $0$'s, while a connected graph with
two distinct eigenvalues is a complete graph $K_n$, $n\geq 2$, whose distinct eigenvalues are $-1$ and $n-1$, see ~\cite[pp. 221]{BH}. Also, by ~\cite[Proposition 1.3.6]{BH}, if $\mathcal{G}$ is a graph with connected components ${\mathcal{G}}_i$ $(1 \leq i \leq s)$, then
the spectrum of $\mathcal{G}$ is the union of the spectra of ${\mathcal{G}}_i$'s (and multiplicities are added). So, the graph $\Delta(G)$ has only one number of distinct eigenvalues if and only if $\Delta(G)$ is the union of $K_1$'s. Since $G$ is a solvable group, Lemma ~\ref{P} implies that $\Delta(G)$ has at most two connected components. Thus in this case, the number of distinct eigenvalues of $\Delta(G)$ is equal to $1$ if and only if $\Delta(G)= K_1$ or $K_1\cup K_1$. Suppose $\Delta(G)\simeq K_{1}$, so there exists a prime number, say $p$, such that $\rho(G)=\{p\}$. Let $\pi(G)=\{p=p_{1},..., p_{n}\}$ and let $A:=\prod_{i=2}^{n}P_{i}$, where  for each $i$, $P_{i}\in Syl_{p_{i}}(G)$. By Ito-Michler's theorem ~\cite[19.10 and 19.11]{Hu1}, $A$ is a normal abelian subgroup of $G$ whose index is a power of $p$. Thus $G=AP$, where $P\in Syl_{p}(G)$. If $P\lhd G$, then we have part $(1)$, otherwise we have part $(2)$. Now consider the case, where $\Delta(G)$ is $K_1\cup K_1$. This implies that $G$ is one of the groups of types one to six, mentioned in ~\cite{ML2}. By ~\cite[Lemma 3.2, 3.3]{ML2}, if $G$ is a group of type two or three, then $\{2\}$ and $\{3\}$ are the connected components of its character graph. As for a group of type six a connected component of the character graph contains at least two vertices, $G$ is not a group of type six. In the sense of ~\cite[Lemma 3.4, 3.5]{ML2}, if $G$ is a group of type five, $\Delta(G)=K_{1}\cup K_{1}$ if and only if $2^{a}+1$ is a prime power, and if $G$ is a group of type four, then $\Delta(G)=K_{1}\cup K_{1}$ if and only if $|E(G)/F(G)|$ and $m$ are distinct prime powers. So we may have either of the cases $(5)$ or $(6)$. Finally, if $G$ is a group of type one, then by ~\cite[Lemma 3.1]{ML2}, $\Delta(G)= K_1\cup K_1$ if and only if $|G/F(G)|$ is a prime power.

Conversely, it is obvious that if $G$ is a group mentioned in $(1)$, then $\Delta(G)$ is an isolated vertex, hence it has only one eigenvalue. Now suppose that we have part $(2)$; Since $N$ is normal abelian, ~\cite[Theorem 6.15]{IS} implies that for each nonlinear $\chi\in {\rm Irr}(G)$, $\chi(1)||G:N|$ where $|G:N|$ is a prime power. Thus $\Delta(G)$ is an isolated vertex and so it has only one eigenvalue. If $G$ is any of the groups mentioned in cases $(3)-(6)$, then by ~\cite[Lemma 3.1, 3.2, 3.3, 3.4, 3.5]{ML2}, one can see that $\Delta(G)$ is the union of two isolated vertices and so it has only one distinct eigenvalue.

Now, suppose that $\Delta(G)$ has two distinct eigenvalues. Since $G$ is a solvable group, by Lemma ~\ref{P}, $\Delta(G)= K_n$ or $K_n\cup K_n$ for some $n\geq 2$. On the other hand, in ~\cite{PP2}, it is proved that for a solvable group $G$, if $\Delta(G)$ has two connected components of orders $n_1$ and $n_2$ with $n_1\leq n_2$, then $n_2\geq 2^{n_1}-1$. Using this, $\Delta(G)\neq K_n\cup K_n$ for each $n\geq 2$. Hence, the number of distinct eigenvalues of $\Delta(G)$ is equal to $2$ if and only if $\Delta(G)= K_n$, where $n\geq 2$.
\end{proof}

\begin{theorem}\label{reg}
Let $G$ be a solvable group such that $\Delta(G)$ is a regular graph with $n$ vertices. $\Delta(G)$ has $3$ distinct eigenvalues if and only if we have the following equivalent properties:
\begin{itemize}
\item[(i)]$n\geq 4$ is even and $\Delta(G)=K_n - M$, where $M$ is a perfect matching. In particular, $\Delta(G)$ is $(n-2)$-regular.
\item[(ii)] $G$ is the direct product of at least two groups with disconnected character graphs of $2$ vertices, such that the prime divisors of character degrees of distinct factors of $G$ are disjoint.
\end{itemize}
\end{theorem}

\begin{proof}
First, we show that $(i)$ and $(ii)$ are equivalent. Note that in any graph, the number of vertices of odd degree is even ~\cite[Corollary 1.2]{BM}. So in each $(n-2)$-regular graph, $n$ is even. Furthermore, in $(n-2)$-regular graphs, each vertex is not adjacent to only one vertex. Thus $(n-2)$-regular graphs are isomorphic to $K_n - M$. By ~\cite[Theorem A]{kayo}, for a solvable group $G$, $\Delta(G)$ is $(n-2)$-regular if and only if $G$ is the direct product of groups with disconnected character graphs of $2$ vertices. Notice that if the character degrees of some distinct factors of $G$ have a common divisor $p$, then by the structure of the character graph of direct product of groups, the vertex $p$ is adjacent to all the other vertices of $\Delta(G)$ (see ~\cite{HIHL}) and this contradicts $(n-2)$-regularity of $\Delta(G)$.

Now, we prove that $\Delta(G)$ has $3$ distinct eigenvalues if and only if $(i)$ holds. Since $G$ is a solvable group, with respect to the results in ~\cite{Z}, regularity of $\Delta(G)$ implies that it is either a complete or $(n-2)$-regular graph (i.e. $\Delta(G) = K_n$ or $K_n - M$ for a perfect matching $M$). Suppose that $\Delta(G)$ has three distinct eigenvalues. This implies that $\Delta(G)$ is not a complete graph, so it is $(n-2)$-regular. We claim that $\Delta(G)$ is connected; if not, then by Lemma ~\ref{P}, it has two connected components which are both complete graphs of orders $n_1$ and $n_2$, where $n=n_{1}+n_{2}$. As $\Delta(G)$ is regular, $n_1 = n_2$. Now by ~\cite{PP2}, $n_2\geq 2^{n_1}-1$ and so $n_{1}=n_{2}=1$. This forces $\Delta(G)$ to have only one distinct eigenvalue, which contradicts our hypothesis. Hence $\Delta(G)$ is a connected graph.
 As a connected regular graph has precisely three distinct eigenvalues if and only if it is strongly regular ~\cite[pp. 213]{BH}), to complete the proof, it is enough to prove that $K_n - M$ is a strongly regular graph. If $n=2$, then $K_n - M = K_1\cup K_1$ with one distinct eigenvalue $0$, which it is not strongly regular, so we may assume that $n\geq 4$. For each pair of adjacent vertices, say $x,y$ in $K_n - M$, all the other vertices in $M$ except their paired vertices are common neighbors of $x$ and $y$. Now consider two non-adjacent vertices $u$ and $v$. These two vertices are paired in $M$. So, all the other $n-2$ vertices are adjacent to both $u$ and $v$. Thus $K_n - M$ is a $(n,n-2,n-4,n-2)$ strongly regular graph, as expected.
\end{proof}

\section{Cut Vertices of Character Graphs}

In this section we discuss the influence of cut vertices on character graphs of solvable groups. Suppose that $G$ is a solvable group. In Theorem $A$ of ~\cite{EI}, the authors discussed a necessary and sufficient condition for $\Delta(G)$ to be Hamiltonian. They proved that $\Delta(G)$ is Hamiltonian if and only if it is a block with at least three vertices. As a block is a connected graph without a cut vertex, it would be interesting to consider the influence of cut vertices on connected character graphs of solvable groups.

\begin{lemma}~\cite[Lemma 18.2]{O}\label{lem:1}
Let $K\lhd G$ and $\frac{N}{K}=F(\frac{G}{K})$. Suppose that $\frac{G}{N}$ is nilpotent. Then we have:
\begin{itemize}
\item[(i)] $q\in\rho(G)$ for all $q\in\pi(\frac{G}{N})$ and they generate a clique in $\Delta(G)$.
\item[(ii)] If $v\in\rho(G)\setminus\pi(\frac{G}{K})$, then either
\begin{itemize}
\item[(a)] $v$ is adjacent to $p$ in $\Delta(G)$, for some $p\in\rho(G)\cap\pi(\frac{N}{K})$; or
\item[(b)] $v$ is adjacent to $q$ in $\Delta(G)$, for all $q\in\pi(\frac{G}{N})$.
\end{itemize}
\end{itemize}
\end{lemma}
\vskip 0.4 true cm

It should be mentioned that during the period of submission of this paper, the authors recognized the paper ~\cite{LCUT} by M.L. Lewis and Q. Meng in which they proved the character graph of a solvable group has at most one cut vertex. Beyond the main idea of the proofs in ~\cite{LCUT} and in this paper which is mainly based on the partitions mentioned in ~\cite{Biss}, in our proofs we strictly discuss the structure of the character graph while it has a cut vertex. Hence we express the following results in complete forms:
\begin{theorem}\label{cutt}
Suppose that $G$ is a solvable group whose character graph is connected and it has a cut vertex of degree two, then $\Delta(G)$ is a path of length two. Furthermore there exists a normal subgroup $K$ of $G$ with $|{\rm cd}(\frac{G}{K})|=2$ and one of the following properties:
 \begin{itemize}
 \item[(i)] $\frac{G}{K}$ is a $p$-group, where $p$ is the cut vertex of $\Delta(G)$.
 \item[(ii)] $\frac{G}{K}$ is a Frobenius group with Frobenius kernel $\frac{N}{K}$, which is an elementary abelian $p$-group for a prime $p\in\rho(G)$ and Frobenius complement $\frac{L}{K}$ whose order is $f$, where ${\rm cd}(\frac{G}{K})=\{1,f\}$. Either $p$ or a prime divisor of $f$ is the cut vertex of $\Delta(G)$.
     \end{itemize}
\end{theorem}

\begin{proof}
Suppose that $v$ is a cut vertex of degree two and $N(v)=\{r,s\}$. Since $v$ is a cut vertex, $r$ and $s$ are in different blocks and so there is no edge between $r$ and $s$. We claim that one of $r$ or $s$ is an endpoint; if not, then there exist two vertices of $\Delta(G)$, say $t_{1}\in N(r)\setminus \{v\}$ and
$t_{2}\in N(s)\setminus \{v\}$. It is clear that $t_{1}\neq t_{2}$ and $t_{1}$ is not adjacent to $t_{2}$. Hence $\{v, t_{1}, t_{2}\}$ does not satisfy Lemma ~\ref{P}, a contradiction.
So without loss of generality we may assume that $r$ is an endpoint. Thus for each $q\in\rho(G)\setminus\{r,s,v\}$, $q$ is adjacent to $s$ and Lemma ~\ref{P} implies that any pair of such vertices are joined by an edge. Therefore $\Delta(G)\setminus\{r,v\}$ is a clique. So we may consider the character graph $\Delta(G)$ as the graph $\Delta(G):\quad r-v-s-K_{l}$, for some integer $l\geq 0$. If $l=0$, then $\Delta(G)$ is a path of length $2$. If $l=1$, then $\Delta(G)$ is a path of length $3$ which is impossible by ~\cite{Zh}. Also if $l=2$, then the diameter of $\Delta(G)$ is $3$ and by
~\cite[Main Theorem]{L3} such graph is not the character graph of a solvable group. Assume that $l\geq 3$, so $|\rho(G)|\geq 6$ and $diam(\Delta(G))=3$. Let $p_{1}$ and $p_{4}$ be those vertices with $d(p_{1},p_{4})=3$. Then $\rho(G)$ has a partition $\rho_{1}\cup\rho_{2}\cup\rho_{3}\cup\rho_{4}$ mentioned in ~\cite{Biss} as follows:
\begin{itemize}
  \item[.] $\rho_{4}=\{q\in\rho(G) : d(p_{1},q)=3\}$,
  \item[.] $\rho_{3}=\{q\in\rho(G) : d(p_{1},q)=2\}$,
  \item[.] $\rho_{2}=\{q\in\rho(G): d(p_{1},q)=d(q,z)=1, \text{for some } z\in\rho_{3}\}$, and
  \item[.] $\rho_{1}=\{p_{1}\}\cup\{q\in\rho(G): d(p_{1},q)=1, d(q,z)\neq 1 \text{ for each } z\in\rho_{3}\}$.
\end{itemize}
Note that we relabel if necessary so that $|\rho_{1}\cup\rho_{2}|\leq |\rho_{3}\cup\rho_{4}|$.
 With respect to the above notation, for $p_{1}=r$ and $p_{4}\in V(K_{l})$ we have $\rho_{1}=\{r\}$, $\rho_{2}=\{v\}$, $\rho_{3}=\{s\}$, and $\rho_{4}=V(K_{l})$. In particular, $|\rho_{3}|=1$ and this is a contradiction with ~\cite[Theorem 2]{Biss} which states $|\rho_{3}|\geq 3$. So there is no solvable group $G$ whose character graph has the above form for $l>0$. Hence we conclude that $\Delta(G)$ has a cut vertex of degree two if and only if it is a path of length two.

Suppose $\Delta(G): w-v-u$. Similar to the proof in ~\cite{H}, we can see that there exists a nontrivial normal subgroup $K$ of $G$ whose quotient is nonabelian and ${\rm cd}(\frac{G}{K})=\{1,f\}$. Applying Gallagher's theorem ~\cite[Corollary 6.17]{IS} we can see that if $\frac{G}{K}$ is a $p$-group for a prime $p$, then $p=v$ which is a cut vertex. Otherwise, ~\cite{IS} implies that $\frac{G}{K}$ is a Frobenius group with Frobenius kernel $\frac{N}{K}$, which is an elementary abelian $p$-group for a prime $p\in\rho(G)$ and Frobenius complement $\frac{L}{K}$ whose order is $f$. Using Gallagher's theorem ~\cite[Corollary 6.17]{IS}, we can see that if $p\notin\rho(G)$, $f$ is not a power of a prime $t$ unless $t=v$, which is a cut vertex.
While $|\pi(f)|=2$, considering different cases for the prime divisors of $f$, we can see that $v$ divides $f$.
Now suppose that $p\in\rho(G)$. Lemma ~\ref{lem:1} yields that either $p$ or a prime divisor of $f$ is the cut vertex of $\Delta(G)$.
\end{proof}

Lewis in \cite[Page 184]{L} asked which graphs with diameter $3$ occur as $\Delta(G)$ for some solvable group $G$.
In the next theorem, we prove that the character graph of any solvable group does not have more than one cut vertex. Specially, we determine the structure of the character graph of a solvable group with at least one cut vertex and diameter $3$.

\begin{theorem}~\label{thm: cut3}
If $G$ is a solvable group whose character graph is connected, then $\Delta(G)$ has at most one cut vertex. In particular, if $v$ is the unique cut vertex of $\Delta(G)$, then $diam(\Delta(G))=3$ if and only if $\Delta(G)$ has one of the following structures:
 \begin{itemize}
 \item[(1)] $\Delta(G)= w-v-K_{s}$, for some $w\in\rho(G)$ and some integer $s\geq 4$ such that $V(K_{s})\cup\{v\}$ does not generate a clique.
 \item[(2)] $\Delta(G)= K_{m}-v-K_{s}$, for some integers $m\geq 8$ and $s\geq 2$, where $V(K_{s})\cup\{v\}$ generates a clique and $1< |N(v)\cap V(K_{m})| < m$.
     \end{itemize}
Furthermore, in both structures $(1)$ and $(2)$, $deg(v)\geq 4$.
\end{theorem}
\begin{proof}
Let $v$ be a cut vertex of $\Delta(G)$ and so $deg(v)>1$. If $deg(v)=2$, then by Theorem~\ref{cutt}, the character graph of $G$ is a path of length $2$ which it has only one cut vertex. Thus we may assume that $deg(v)\geq 3$. If $\Delta(G)-v$ has at least three components, then by choosing a vertex from each components, we find three vertices such that there is no edge between any pair of them, which contradicts solvability of $G$ by Lemma ~\ref{P}. So $\Delta(G)-v$ has two components with vertex sets $C_1$ and $C_2$, respectively. By Lemma ~\ref{P}, one can see that $C_1$ and $C_2$ are complete graphs. So, we may have the following cases:
\begin{itemize}
\item[(i)] Either $|C_1\cap N(v)|=1$ or $|C_2\cap N(v)|=1$.
\item[(ii)] $|C_i\cap N(v)|\geq 2$ for $i=1, 2$.
\end{itemize}
Suppose that we have the case (i). Without loss of generality, we may assume that $|C_{1}\cap N(v)|=1$. Let $w\in C_{1}\cap N(v)$. For each $t\in C_{1}\setminus\{w\}$ and for each $l\in C_{2}$, it is clear that $d(t,l)=3$ which implies that $diam(\Delta(G))=3$. Suppose that $C_{1}\setminus\{w\}\neq\emptyset$. Let $p_{1}$ be a vertex of $C_{1}\setminus\{w\}$ and $p_{4}$ be a vertex of $C_{2}$. Since $d(p_{1},p_{4})=3$, $\rho(G)$ can be considered as the following disjoint union:
$$\rho(G)=\rho_{1}\cup\rho_{2}\cup\rho_{3}\cup\rho_{4},$$
where for each $i$, the subset $\rho_{i}$ is defined as in the proof of Theorem~\ref{cutt}. Considering the structure of $\Delta(G)$,  $\rho_{3}$ is one of the sets $\{v\}$ or $\{w\}$. Thus $|\rho_{3}|=1$ which contradicts ~\cite[Theorem 2]{Biss}, hence $C_{1}\setminus\{w\}=\emptyset$. This implies that $w$ is an end point and so is not a cut vertex. Also, since $deg(v)\geq 3$, $|C_2\cap N(v)|\geq 2$ and so $v$ is the unique cut vertex of $\Delta(G)$.

 In this case, if $diam(\Delta(G))=3$, then there exists $l\in C_{2}\setminus N(v)$. Since the diameter of character graph of a solvable group with order $5$ is at most $2$ (see ~\cite[Main Theorem]{ML3}), we have $|C_{2}|\geq 4$. Now with respect to the above notation, for $p_{1}:=w$ and $p_{4}:=l$, we have $\rho_{4}=C_{2}\setminus N(v)$, $\rho_{3}=N(v)\setminus\{w\}$, $\rho_{2}=\{v\}$, and $\rho_{1}=\{w\}$. By ~\cite[Theorem 2]{Biss},  $deg(v) = |\rho_{3}|+1\geq 4$. So $\Delta(G)$ has the structure $(1)$ in the theorem.

Considering case (ii), it is easy to see that $v$ is the unique cut vertex of $\Delta(G)$ with degree at least $4$. In this case, $diam(\Delta(G))=3$ if and only if there exists $t\in C_{i}\setminus N(v)$, for some $i\in\{1,2\}$. Without loss of generality, assume that $i=1$. By Lemma ~\ref{P} we can see that $C_{2}\cup\{v\}$ generates a clique. Let $l\in C_{2}$. Using the above notation, for $p_{1}:=l$ and $p_{4}:=t$, we have $\rho_{4}=C_{1}\setminus N(v)$, $\rho_{3}=C_{1}\cap N(v)$, $\rho_{2}=\{v\}$, and $\rho_{1}=C_{2}$. By ~\cite[Theorem 2]{Biss}, $|\rho_{3}| = |C_{1}\cap N(v)|\geq 3$. Since we started with the assumption $t\in C_{1}\setminus N(v)$, it is concluded that $|C_{1}|\geq 4$. Notice that $|\rho_{1}\cup \rho_{2}|=|C_2|+1\geq |N(v)\cap C_2|+1\geq 3$ and $|\rho_{3}\cup \rho_{4}|=|C_1|$. So by ~\cite[Theorem 4]{Biss}, $|C_1|\geq 2^{3}$. Thus for $m:=|C_{1}|$ and $s:=|C_{2}|$, we have the structure $(2)$ in theorem for $\Delta(G)$.

It should be mentioned that if we have either of the structures $(1)$ or $(2)$, then $diam(\Delta(G))=3$.
\end{proof}

Now, we propose the following question.

\begin{question}
In ~\cite{ML4}, Lewis has constructed a solvable group whose character graph has diameter $3$. This graph has the structure $(1)$ in Theorem ~\ref{thm: cut3} and its order is $6$. So, it has the smallest possible order. The minimum order of a character graph of a solvable group with a cut vertex and diameter $3$ which has the structure $(2)$ in Theorem ~\ref{thm: cut3}, is at least $11$. Does there exist a finite solvable group with such a character graph?
\end{question}

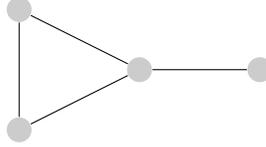
\begin{figure}[h]
\begin{center}
\begin{tikzpicture}
  [scale=.8,auto=left,every node/.style={circle,fill=black!20}]
  \node (n1) at (1,1)  {};
  \node (n2) at (1,3)  {};
  \node (n3) at (3,2)  {};
  \node (n4) at (5,2)  {};

  \foreach \from/\to in {n1/n2,n1/n3,n2/n3,n3/n4}
    \draw (\from) -- (\to);
\end{tikzpicture}
\end{center}
\caption{a connected graph with four vertices and one triangle.}
~\label{fig:10}
\end{figure}

\begin{theorem}~\label{cut}
Let $G$ be a solvable group whose character graph is a connected graph of order $n$ with a cut vertex $r$. We have the following properties:
 \begin{itemize}
 \item[(i)] If $\Delta(G)$ is $K_{4}$-free and $G$ contains a normal $r$-complement, then $n\leq 4$ and $\Delta(G)$ is either a path of length two or the graph in Figure ~\ref{fig:10}.
     \item[(ii)] If $\Delta(G)$ is $K_{n-1}$-free, $r\notin\rho(F(G))$ and $F(G)$ is nonabelian, then $|\rho(F(G))|\leq n-2$. In particular, if $|\rho(F(G))|=n-2$, then $n\geq 6$.
\end{itemize}

\end{theorem}
\begin{proof}
 First suppose that $\Delta(G)$ is $K_{4}$-free and $G$ has a normal $r$-complement $H$. This implies that $\rho(H)=\rho(G)\setminus\{r\}$. Let $\chi$ be a nonlinear irreducible character of $G$. If $r$ does not divide $\chi(1)$, then $gcd(\chi(1),[G:H])=1$, hence $\chi_{H}\in {\rm Irr}(H)$, and so $\chi(1)\beta(1)\in {\rm cd}(G)$, for each $\beta(1)\in {\rm cd}(G/H)$. Therefore $r$ is joined to all vertices of $\Delta(G)$. As $r$ is a cut vertex of $\Delta(G)$ and $G$ is solvable, $\Delta(H)$ is a disconnected graph with two complete connected components. Since $\Delta(G)$ is $K_{4}$-free and $r$ is a neighbor of all vertices of $\Delta(G)$, each connected component of $\Delta(H)$, has at most two vertices. This implies that $n\leq 5$. The case $n=5$ is not possible since otherwise, solvability of $H$ yields that $\Delta(H)$ is not of the form $K_{2}\cup K_{2}$ ~\cite{PP2}. Thus $n\leq 4$. One can see that $n=4$ if and only if $\Delta(G)$ is the graph in Figure ~\ref{fig:10}, and $n=3$ if and only if it is a path of length two.

Consider the second case mentioned in the theorem. Note
that if $N$ is a normal subgroup of $G$ with
$\rho(N)=\rho(G)\setminus\{r\}$, then $\Delta(N)$ is a
disconnected graph with two complete components. As $F(G)$ is
nilpotent, $\Delta(F(G))$ is a complete graph. Thus if $F(G)$ is
non-abelian and $r\notin\rho(F(G))$, then as we discussed
$|\rho(F(G))|\leq n-2$.  In particular suppose that $\Delta(F(G))\simeq
K_{n-2}$. Since $r$ is a cut vertex of $\Delta(G)$ and $G$ is a solvable group, we can see that, for a prime $p$, $\Delta(G)$ has
the following form: $$p-r-K_{n-2}.$$
 As $\Delta(G)$ is $K_{n-1}$-free, there exists $t\in V(K_{n-2})$ such that $t$ is not adjacent to $r$. This implies that the diameter of $\Delta(G)$ is $3$, hence by ~\cite{L3}, we have $n\geq 6$.
\end{proof}

\begin{theorem}
Suppose $G$ is a solvable group whose $\Delta(G)$ is a connected graph with $|\rho(G)|\geq 4$. If we have one of the following cases, then $\Delta(G)$ is hamiltonian:
\begin{itemize}
\item[(i)] $\Delta(G)$ is a non-regular graph without a complete vertex and $F(G)$ is abelian.

\item[(ii)] $\Delta(G)$ is a non-complete regular graph.\end{itemize}
\end{theorem}
\begin{proof}
If $\Delta(G)$ is a non-regular graph without a complete vertex and $F(G)$ is abelian, then
~\cite[Theorem A]{ZUCC} verifies that $G\simeq D_{1}\times...\times D_{n}$, where $\Delta(D_{i})$'s are disconnected graphs. Connectedness of $\Delta(G)$ verifies that $i\geq 2$. Now By ~\cite[Corollary C]{EI}, we conclude that $\Delta(G)$ is Hamiltonian.

On the other hand, if $\Delta(G)$ is a non-complete regular graph, then ~\cite[Theorem A]{kayo} implies that $G$ is the direct product of groups with disconnected character graphs of $2$ vertices. Similar to the previous case we can see that $\Delta(G)$ is Hamiltonian.
\end{proof}

\subsection*{Acknowledgement}
In this paper, S. Hossein-Zadeh is supported by National Elites
Foundation under Grant Number $15/93173$.

\end{document}